\documentclass{birkjour}

 \newtheorem{thm}{Theorem}[section]
 \newtheorem{cor}[thm]{Corollary}
 
 \newtheorem{prop}[thm]{Proposition}
 \theoremstyle{definition}
 \newtheorem{defn}[thm]{Definition}
 \theoremstyle{remark}
 \newtheorem{rem}[thm]{Remark}
 
 \numberwithin{equation}{section}
\usepackage{cite}
\usepackage{color}
\begin{document}

\title[Quaternionic proportional fractional Fueter operator]{A quaternionic proportional fractional\\ Fueter-type operator calculus}

\author[J. O. Gonz\'alez-Cervantes]{Jos\'e Oscar Gonz\'alez-Cervantes}

\address{Departamento de Matem\'aticas, ESFM-Instituto Polit\'ecnico Nacional.\\
07338\\
Ciudad M\'exico, M\'exico}
\email{jogc200678@gmail.com}

\author{Juan Bory-Reyes}
\address{SEPI-ESIME-Zacatenco-Instituto Polit\'ecnico Nacional.\br
07338\br
Ciudad M\'exico, M\'exico}
\email{juanboryreyes@yahoo.com}



\begin{abstract}
The main goal of this paper is to construct a proportional analogues of the quaternionic  fractional Fueter-type operator recently introduced in the literature. We start by establishing a quaternionic version of the well-known proportional fractional integral and derivative with respect to a real-valued function via the Riemann-Liouville fractional derivative. As a main result, we prove a quaternionic proportional fractional Borel-Pompeiu formula based on a quaternionic proportional fractional Stokes formula. This tool in hand allows us to present a Cauchy integral type formula for the introduced quaternionic proportional fractional Fueter-type operator with respect to a real-valued function. 
\end{abstract}

\maketitle

\section{Introduction} 
Quaternionic analysis, as an extension of the classical theory of complex holomorphic functions, focuses on the connection between analysis, geometry in $\mathbb R^4$ (even $\mathbb R^3$) and the algebraic structure of quaternions. Its tools can be applied to several different areas of physics, natural sciences and engineering. With no claim of completeness, for a more extended introduction to quaternionic analysis  we refer to the books \cite{GS1, GS2, GHS, Krav, KS} and the references therein. However, the classical works on the subject go back as far as \cite{De, F, Ha, MT, sudbery}.

Fractional calculus, which addresses the derivatives and integrals with any order, has attracted an increasing attention turning out to be useful in mathematical modeling of social and natural systems as well as engineering problems of the real world \cite{D, H, K, KST,  OT, P, SKM, Ro}.

In the literature, one can find many variants of fractional operators (see \cite{AJ, B, KAYS, K, Ka1, Ka2, AB} amongst others). Nevertheless, in several applications, the distinct approaches lead to identical results or can be related directly or indirectly in a unified framework, see for instance \cite{FU, OBT, OB}.

One approach to the last issue focus on the tempered fractional calculus model, which has been a subject of investigation in recent years for a number of its applications to many areas of applied science and engineering. The first description of the tempered fractional integral seems to go back as far as \cite{AMS, Bu, Pi}, but a very readable and more precisely account of the associated model of fractional calculus has been used successfully in \cite{LDZ, MSC}. 

The topic of tempered fractional calculus has been recently re-discovered under the name of proportional (also generalized proportional) fractional calculus \cite{JARH, JAAl, JAA, JUAB}, which is essentially equivalent to the tempered framework via the identities appear in \cite[Remark 2.1]{FU}. 

However, the possibility that the second parameter in the tempered fractional calculus to be any complex number certainly makes it a more general model, although the use of complex parameter is of limited interest and more complicated to interpret.

The calculus of classical fractional integration and differentiation of a function with respect to another function was developed in the standard textbooks of \cite[§18.2]{SKM} and \cite[§2.5]{KST}. For a fuller treatment of the more general case of proportional (tempered) framework we refer the reader to \cite{JARH, JAAl, JAA, JUAB, FFUS, MKFF} and the references given there.

The development of links between quaternionic analysis (or more general Clifford analysis) and conventional fractional calculus (related to Riemann-Liouville and Caputo approaches) represents a very recent topic of research. In particular, some first steps to introduce a fractional hyperholomorphic function theory have been made in the works \cite{BP,BPP, DM, GB, PAAB, YX, BG2, FRV, KV, V} and the references therein. In these papers, the authors established a fractional integro-differential operator calculus for quaternion-valued functions. 

The current work is a continuation of the previous one \cite{BG2}, where a fractional generalized Fueter operator with respect to a vector-valued function was first introduced. Now we develop a quaternionic proportional fractional Fueter-type operator calculus with respect to a real-valued function via Riemann-Liouville fractional derivative in quaternionic analysis context. As a main result, we prove a quaternionic proportional fractional Borel-Pompeiu formula based on a quaternionic proportional fractional Stokes formula. 

Motivated by \cite{FRV}, we touch the duality phenomenon of the Stokes and Borel-Pompeiu formulas, which manifests itself between Caputo and Riemann-Liouville fractional derivatives, as well as between left and right fractional derivatives. 

The structure of the work reads as follows: After this brief introduction Preliminary's section contains a summary of basic facts about quaternionic analysis, fractional derivatives and their main properties. In Section 3 our main results are stated and proved.

\section{Preliminaries} 
In this section, we present some basic concepts of fractional proportional operators as well as rudiments of quaternionic analysis employed throughout the work. The material provided can be found in \cite{JARH, JAAl, JAA, JUAB} and \cite{MS, S1, shapiro1, shapiro2} respectively.

An overlap between these two subjects, fractional proportional operators and quaternionic analysis, is to be discussed in the present work.

\subsection{Proportional fractional integrals and derivatives of a function with\\ respect to another function}
Consider the continuous functions $ \chi_0 ,  \chi_1: [0, 1]\times \mathbb R\to [0,\infty) $ such that  
$$\lim_{\sigma\to 0^+} \chi_1(\sigma,t) = 1,  \lim_{\sigma\to 0^+} \chi_0(\sigma,t) = 0, \lim_{\sigma\to 1^-} \chi_1(\sigma,t) = 0, \lim_{\sigma\to 1^-} \chi_0(\sigma,t) = 1.$$
The proportional derivative of $f\in C^1(\mathbb R)$ of order $\sigma\in [0,1]$ is given by
\begin{align*} 
D^{\sigma}f (t) =  \chi_1(\sigma,t)f (t) +  \chi_0(\sigma,t)f'(t).  
\end{align*}

In particular, for $ \chi_1(\sigma, t) = 1- \sigma$  and $\chi_0(\sigma, t) = \sigma$  we have 
\begin{align}\label{equa1}
(D^{\sigma}f )(t) = & (1-\sigma)f (t) +  \sigma f'(t), \\
({}_aI^{1-\sigma} f) (t) = &\frac{1}{\sigma}\int_
a^t  e^{ \frac{\sigma-1}{\sigma}(t-s)}f (s) ds, \nonumber  
 \end{align}
What is more, given  $\varphi\in C^1(\mathbb R )$ such that  $\varphi '(t)>0 $, for all $t\in \mathbb R$,  then the  
proportional derivative of $f\in C^1(\mathbb R)$ with respect to $\varphi$ of order $\sigma\in [0,1]$, is defined to be
\begin{align}\label{DerPropRespecFuntL}
(D^{\sigma,\varphi} f )(t) = (1-\sigma)f(t) + \sigma\frac{f'(t)}{\varphi'(t)}.
\end{align}
The proportional integral of $f$  with respect to $\varphi$  and order $n \in \mathbb N$ reads   
\begin{align}\label{IntPropRespecFuntL}
 ({}_a I^{n,\sigma,\varphi} f) (t) = \frac{1}{\sigma^n \Gamma(n)}
 \int_a^t
e^{\frac{\sigma-1}{\sigma} (\varphi(t)-  \varphi(\tau)) }
(\varphi(t)- \varphi(\tau))
^{n-1}f (\tau)\varphi'(\tau) d\tau.
\end{align}
 
Note that defining 
\begin{align}\label{AxuliarR}
 (E^{\sigma,\varphi} f )(t) := \varphi'(t) \sigma^{-1}(1-\sigma)f(t) +  f'(t) .
\end{align}
then we see that  
\begin{align}\label{AxuliarComputationsR}
\frac{\partial }{\partial t}(e^{\varphi(t) \sigma^{-1}(1-\sigma)   }f(t)) = e^{\varphi(t) \sigma^{-1}(1-\sigma)  } ( E^{\sigma,\varphi} f )(t) .
\end{align}
The previous identities describe how the usual derivative $\displaystyle\frac{\partial }{\partial t}$ relates to $D^{\sigma,\varphi}$. 

Set $\alpha\in \mathbb C$, with $0< Re(\alpha) <1 $.  The left and right proportional fractional integrals, in the Riemann-Liouville setting, of $f\in AC^1([a,b])$ with order $\alpha$ are defined by
\begin{align}\label{IntPropFracL}
({}_aI^{\alpha,\sigma} f) (t) = \frac{1}{\sigma^{\alpha} \Gamma(\alpha)}\int_a^t e^{ \frac{\sigma-1}{\sigma}  (t-\tau)} (t-\tau)^{ \alpha-1} f(\tau) d\tau \end{align}
and \begin{align}\label{IntPropFracR}
( I_b^{\alpha,\sigma} f) (t) = \frac{1}{\sigma^{\alpha} \Gamma(\alpha)}\int_ t^b e^{ \frac{\sigma-1}{\sigma}  ( \tau-t)} ( \tau-t)^{ \alpha-1} f(\tau) d\tau \end{align}
respectively.  

The left and right proportional derivatives of $f$ in  Riemann-Liouville sense are defined by
\begin{align}\label{DerPropFracL}
({}_aD^{\alpha,\sigma} f) (t) = D^{ \sigma} {}_aI^{1-\alpha,\sigma} f(t)
  \quad \textrm{and} \quad 
( D_b^{\alpha,\sigma} f) (t) = D^{\sigma}  I_b^{1-\alpha,\sigma} f(t),
\end{align}
and the left and right proportional derivatives of $f$ in  Caputo  sense are defined by
\begin{align}\label{DerPropFracLC}
({}^C _aD^{\alpha,\sigma} f) (t) ={}_aI^{1-\alpha,\sigma}  D^{ \sigma} f(t)
  \quad \textrm{and} \quad 
( {}^CD_b^{\alpha,\sigma} f) (t) = I_b^{1-\alpha,\sigma}  D^{\sigma}  f(t),
\end{align}

In addition, the left fractional  proportional integral with respect to $\varphi$ is given by
\begin{align}\label{IntFracPropRespecFuntL}  
 ({}_a I^{\alpha,\sigma,\varphi} f) (t) := \frac{1}{\sigma^\alpha \Gamma(\alpha)}
 \int_a^t
e^{\frac{\sigma-1}{\sigma} (\varphi(t)-\varphi(\tau)) }
(\varphi(t)- \varphi(\tau))
^{\alpha-1}f (\tau)\varphi'(\tau) d\tau, 
\end{align}  
while the right fractional proportional integral reads
\begin{align}\label{IntFracPropRespecFuntR}
(I_b^{\alpha,\sigma,\varphi} f)(t) := \frac{1}{\sigma^\alpha \Gamma(\alpha)}
 \int_ t^b
e^{\frac{\sigma-1}{\sigma} (\varphi(\tau)- \varphi(t)) }
(\varphi (\tau)-  \varphi (t))
^{\alpha-1}f (\tau)\varphi'(\tau) d\tau. 
\end{align}
If $1 \leq \Re \alpha $  and $n = [\Re  \alpha] + 1$ we can introduce the left and the right fractional proportional derivatives with respect to $\varphi$ as follow
\begin{align}\label{DerFracPropRespecFuntL}
({}_aD^{\alpha, \sigma,\varphi} f)(t) :=  D^{n,\sigma,\varphi}
{}_aI^{n-\alpha,\sigma,\varphi}f(t) ,  \quad
(D_b^{\alpha, \sigma,\varphi} f)(t) := D^{n,\sigma,\varphi}
 I_b^{n-\alpha,\sigma,\varphi}f(t),
\end{align}
where $D^{n,\sigma, \varphi}:=\underbrace{D^{\sigma, \varphi} \circ \cdots \circ D^{\sigma, \varphi}}_\text{n-times}$. 

Moreover, 
\begin{align}\label{FundTheoFPW}
{}_aD^{\alpha,\sigma,\varphi} \circ 
{}_aI^{\alpha,\sigma,\varphi} f(t) = f(t) \quad \textrm{and} \quad
 D_b^{\alpha,\sigma,\varphi} \circ 
 I_b^{\alpha,\sigma,\varphi} f(t) = f(t).
\end{align}

Clearly, if    $\sigma= 1$ from previous definitions we obtain 
\begin{itemize}
\item   Riemann-Liouville fractional derivatives and integrals,  if  $\varphi(t) = t$ for all $t$.
\item The fractional derivatives and integrals  in the Katugampola setting,  if     $\varphi(t) =\frac{ t^{\mu}}{\mu}$ for all $t$.
\item   Hadamard fractional derivatives and integrals, if  $\varphi(t) =\ln(t)$ for all $t>0$.
\end{itemize}
Finally, in \cite[Definition 4.4.]{JAA} we see that the left and the right- Caputo fractional proportional derivative of a function with
respect to $\varphi$ is given by 
\begin{align}\label{CaputoDer}
({}^{C}_aD^{\alpha, \sigma, \varphi} f)(t) := ({}_a I^{1-\alpha,\sigma,\varphi} ) (D^{\sigma,\varphi} f )(t)
\end{align}
and 
\begin{align*}
({}^{C} D_b^{\alpha, \sigma, \varphi} f)(t) := ( I_b^{1-\alpha,\sigma,\varphi} ) (D^{\sigma,\varphi} f )(t).
\end{align*}

\subsection{Brief summary of the quaternionic analysis}
The skew-field of quaternions $\mathbb H$  with basis $\{1, {\bf i}, {\bf j}, {\bf k}\}$ is a four dimensional associative division algebra over $\mathbb R$. 

Any quaternion $x$ may be represented by
$$x=x_0 1+x_{1} {\bf i}+x_{2} {\bf j}+x_{3} {\bf k},$$ 
where $x_{k}\in \mathbb R, k= 0,1,2,3$.

Addition and multiplication by a real scalar is defined in component-wise manner. The multiplication distributes over addition, $1$ is
the multiplication identity, and
$$ {\bf i}^{2}= {\bf j}^{2}= {\bf k}^{2}=-1,$$
$$ {\bf i}\, {\bf j}=- {\bf j}\,\  {\bf i}= {\bf k};\   {\bf j}\, {\bf k}=- {\bf k}\, {\bf j}= {\bf i};\  {\bf k}\, {\bf i}=- {\bf i}\, {\bf k}= {\bf j}.$$ 
On $\mathbb H$ a mapping conjugation is defined as $x\rightarrow {\overline x}:=x_0-x_{1}{\bf i}-x_{2}{\bf j}-x_{3}{\bf k}$. By the length or
norm of $x\in \mathbb H$, denoted by $|x|$, we mean
$$|x|^{2} := x^{2}_{0}+x^{2}_{1}+x^{2}_{2}+x^{2}_{3}= x\,{\overline x}={\overline x}\,x.$$

The quaternionic scalar product of two arbitrary elements $q, x\in\mathbb H$ is given by 
$$\langle q, x\rangle:=\frac{1}{2}(\bar q x + \bar x q) = \frac{1}{2}(q \bar x + x  \bar q).$$

An ordered set of quaternions $\psi:=\{\psi_0, \psi_1,\psi_2,\psi_3\}$ that satisfies the condition of orthonormality
$$\langle \psi_k, \psi_s\rangle =\delta_{k,s},\ k,s=0,1,2,3,$$ 
($\delta_{k,s}$ is Kronecker's symbol) is called a structural set.

For fixed structural set $\psi$, any quaternion $x$ can be rewritten as 
$$x_{\psi} := \sum_{k=0}^3 x_k\psi_k,$$ 
where $x_k\in\mathbb R$ for all $k$. The mapping
\begin{equation} \label{mapping}
\sum_{k=0}^3 x_k\psi_k \rightarrow (x_0,x_1,x_2,x_3).
\end{equation}
will used in essential way.

Given $q, x\in \mathbb H$, let us introduce the notation $\langle q, x \rangle_{\psi}$ for $\sum_{k=0}^3 q_k x_k.$

In what follows, $\Omega\subset\mathbb H\cong \mathbb R^4$ denotes an open bounded domain with boundary $\partial \Omega$ a 3-dimensional smooth surface. The closure of $\Omega$ will be denoted by $\overline\Omega$.
 
We will consider functions $f$ defined in $\Omega$ with value in $\mathbb H$. They may be written as: $f=\sum_{k=0}^3 f_k \psi_k$, where $f_k, k= 0,1,2,3,$ are $\mathbb R$-valued functions in $\Omega$. Properties as continuity, differentiability, integrability and so on, which as ascribed to $f$ have to be posed by all components $f_k$. We will follow standard notation, for example $C^{1}(\Omega, \mathbb H)$ denotes the set of continuously differentiable $\mathbb H$-valued functions defined in $\Omega$. 
 
Every structural set generates left- and right- $\psi-$Fueter operators acting on $f, g \in C^1(\Omega,\mathbb H)$, which are defined respectively by   
$${}^{{\psi}}\mathcal D[f] := \sum_{k=0}^3 \psi_k \partial_k f$$ 
and 
$${}^{{\psi}}\mathcal  D_r[g] :=  \sum_{k=0}^3 \partial_k g \psi_k,$$ 
where $\partial_k f :=\displaystyle \frac{\partial f}{\partial x_k}$ for all $k$. 

Respective solutions $f, g$ of ${}^{{\psi}}\mathcal D[f]=0$ (${}^{{\psi}}\mathcal  D_r[g]=0$) in  $\Omega\subset\mathbb H$  will be referred to as left (right) $\psi$-hyperholomorphic functions.

The Stokes'formula for the $\psi$-hyperholomorphic functions theory reads 
\begin{align}\label{StokesHyp} \int_{\partial \Omega} { g }\nu^\psi_x f =  &   \int_{\Omega } \left( g 
{}^\psi \mathcal  D[f] + {}^{{\psi}}\mathcal  D_r[g] f\right)dx,
\end{align}
for all $f,g \in C^1(\overline{\Omega}, \mathbb H)$. Here $dx$ denotes the differential form of 4-dimensional volume in $\mathbb R^4$ and  
$$\nu^{{\psi} }_{x}:=-sgn\psi \left( \sum_{k=0}^3 (-1)^k \psi_k d\hat{x}_k\right)$$ 
stands for the differential form of 3-dimensional surface in $\mathbb R^4$ according to $\psi$, 
where $d\hat{x}_k$ is is the differential form $dx_0 \wedge dx_1\wedge dx_2  \wedge  dx_3$ with the factor $dx_k$ omitted. In addition, the value of $sgn\psi$ is $1$, or $-1$, if $\psi$ and $\psi_{std}:=\{1, {\bf i}, {\bf j}, {\bf k}\}$ have the same orientation or not, respectively.  
 
Let us recall that the $\psi$-hyperholomorphic Cauchy Kernel is given by 
 \[ K_{\psi}(y- x)=\frac{1}{2\pi^2} \frac{ \overline{y_{\psi} - x_{\psi}}}{|y_{\psi} - x_{\psi}|^4}.\] 
The Borel-Pompieu integral formula shows that
\begin{align}\label{BorelHyp}  &  \int_{\partial \Omega}(K_{\psi}(y-x)\nu_{y}^{\psi} f(y)  +  g(y)   \nu_{y}^{\psi} K_{\psi}(y-x) ) \nonumber  \\ 
&  - 
\int_{\Omega} (K_{\psi} (y-x) {}^{\psi}\mathcal D [f] (y) + {}^{{\psi}}\mathcal  D_r [g] (y) K_{\psi} (y-x) )dy   \nonumber \\
		=  &  \left\{ \begin{array}{ll}  f(x) + g(x) , &  x\in \Omega,  \\ 0 , &  x\in \mathbb H\setminus\overline{\Omega}.                     
\end{array} \right. 
\end{align} 
\section{Main results} 
Let $\vec{\alpha}:= (\alpha_0, \alpha_1,\alpha_2,\alpha_3),\ \vec{\beta}:=(\beta_0, \beta_1, \beta_2, \beta_3) \in \mathbb C^4$. Set $a:=\sum_{k=0}^3\psi_k a_k,\ b:=\sum_{k=0}^3\psi_k b_k,\ \sigma:=\sum_{k=0}^3\psi_k \sigma_k \in \mathbb H$ such that $a_k< b_k$, $0< \Re \alpha_k, \Re\beta_k \ \mbox{and} \ \sigma_k < 1$ for $k=0,1,2,3$. 

Write 
${J_a^b }:= \{  \sum_{k=0}^3\psi_k x_k \in \mathbb H \ \mid \ a_k\leq   x_k \leq   b_k,\ k=0,1,2,3\}= [a_0,b_0] \times [a_1,b_1] \times [a_2,b_2]  \times [a_3,b_3].$ 

By ${f}=\sum_{i=0}^3\psi_i f_i\in AC^1(J_a^b,\mathbb H)$ we mean that mapping  
$$x_j \mapsto f_i(q_0,\dots,x_j,\dots, q_3)$$ 
belongs to $AC^1([a_j, b_j], \mathbb R)$ for each $q\in J_a^b$ and all $i, j=0,1,2,3$. 

\subsection{Quaternionic proportional fractional Fueter-type operator in Riemann-Louville and Caputo sense} 
In our analysis, the structure of the integral and differential operators given in \eqref{IntPropFracL} and  \eqref{IntPropFracR} are preserved. Let ${f}\in AC^1(J_a^b,\mathbb H)$ then the left and the right fractional proportional integrals of $f$ by coordinates
with order $\alpha$ and proportion associated to $\sigma$ are given by
\begin{align*} 
 ({}_a{\mathcal I}^{\alpha,\sigma}  f )(x,q)  = & \sum_{i=0}^3 {}_{a_i}I^{\alpha_i,\sigma_i} f (q_0,\dots, x_i, \dots q_3) \\
= & \sum_{j=0}^3 \psi_j  \sum_{i=0}^3   \frac{1}{\sigma_i^{\alpha_i}\Gamma(   \alpha_i    )} \int_{a_i}^{x_i}\frac{ e^{\frac{\sigma_i- 1}{\sigma_i}(x_i-\tau_i)}   } {  (x_i - \tau _i)^{1-   {  \alpha_i }  }  } f_j  
(q_0, \dots,  \tau _i, \dots, q_3)  d \tau_i.
\end{align*}
and 
\begin{align*} 
({\mathcal I}_b^{\alpha,\sigma}  f )(x,q)  = & \sum_{i=0}^3 I_{b_i} ^{\alpha_i,\sigma_i} f (q_0,\dots, x_i, \dots q_3).
\end{align*}
Let $f\in C^1(J_a^b, \mathbb H)$. To extend operator \eqref{equa1}, set
\begin{align*}
{}^{\psi }{\mathcal D}^{\sigma} f := (1-\sigma) f + \sigma {}^{\psi}\mathcal D f, \\
{}^{\psi }{\mathcal D_r}^{\sigma} f :=  f (1-\sigma) + {}^{\psi}\mathcal D_r f \sigma.
\end{align*}
\begin{defn} 
The  left and the right fractional proportional derivative of ${f} \in AC^1(J_a^b,\mathbb H)$ in the Riemann-Liuoville sense, on the left, with order $\alpha$ and proportion associated to $\sigma$ are defined to be  
\begin{align*} 
({}^{\psi}_a{\mathcal D}^{\alpha,\sigma}  f )(x,q)  = {}^{\psi }{\mathcal D}^{\sigma}   ({}_a{\mathcal I}^{1-\alpha,\sigma}  f )(x,q) 
\end{align*}
and 
\begin{align*} 
({}^{\psi}{\mathcal D}_b^{\alpha,\sigma}  f )(x,q)  = {}^{\psi }{\mathcal D}^{\sigma}   ( {\mathcal I}_b^{1-\alpha,\sigma}  f )(x,q), 
\end{align*}
where we let $1-\alpha$ stand for $(1-\alpha_0, 1-\alpha_1, 1-\alpha_2, 1-\alpha_3)\in \mathbb C^4 $ and $x$ indicates the differentiation quaternionic variable of ${}^{\psi }{\mathcal D}^{\sigma}$.  

Similarly, the  left  and the right fractional proportional derivative of $f \in AC^1(J_a^b,\mathbb H)$, on the right,  with  order $\alpha$ and proportion associated to $\sigma$ are   
\begin{align*} 
({}^{\psi}_a{\mathcal D}_r^{\alpha,\sigma}f)(x,q) = {}^{\psi }{\mathcal D}_r^{\sigma}({}_a{\mathcal I}^{1-\alpha,\sigma}f)(x,q) 
, \quad ({}^{\psi}{\mathcal D}_{r,b}^{\alpha,\sigma}f)(x,q)  = {}^{\psi }{\mathcal D}_r^{\sigma}({\mathcal I}_b^{1-\alpha,\sigma}f)(x,q). 
\end{align*}
\end{defn}

\begin{defn} 
The  left and the right fractional proportional derivative of $f\in AC^1(J_a^b,\mathbb H)$ in the Caputo sense, on the left, with order $\alpha$ and proportion associated to $\sigma$ are  
\begin{align*} 
({}^{\psi,C}_a{\mathcal D}^{\alpha,\sigma}  f )(x,q)  = ({}_a{\mathcal I}^{1-\alpha,\sigma}\  {}^{\psi }{\mathcal D}^{\sigma}   f )(x,q) 
\end{align*}
and 
\begin{align*} 
({}^{\psi,C}{\mathcal D}_b^{\alpha,\sigma}  f )(x,q)  =   ( {\mathcal I}_b^{1-\alpha,\sigma} \ {}^{\psi }{\mathcal D}^{\sigma}  f )(x,q) 
\end{align*}
and its right versions are 
\begin{align*} 
({}^{\psi,C}_a{\mathcal D}_r^{\alpha,\sigma}  f )(x,q)  = &  ({}_a{\mathcal I}^{1-\alpha,\sigma}\  {}^{\psi }{\mathcal D}_r^{\sigma}  f )(x,q), \\ 
 ({}^{\psi,C}{\mathcal D}_{r,b}^{\alpha,\sigma}f)(x,q) = & ({\mathcal I}_b^{1-\alpha,\sigma} \  {}^{\psi }{\mathcal D}_r^{\sigma}f)(x,q). 
\end{align*}
\end{defn}

Note that ${}^{\psi}_a{\mathcal D}^{\alpha,\sigma}$ and  ${}^{\psi}{\mathcal D}_b^{\alpha,\sigma}$ are natural extensions of ${}_a D^{\alpha,\sigma}$  and $D_b^{\alpha,\sigma} $ given in \eqref{DerPropFracL}, respectively. The same behaviour for quaternionic fractional proportional operators in Caputo sense  \eqref{DerPropFracLC}.

\begin{prop}\label{DFracPropor} Suppose $0<\sigma _k, \rho_k<1$ for $k=0,1,2,3$. Denote $\rho =\sum_{k=0}^3 \psi_k\rho_k $, $\rho^{-1} (1-\rho )= \gamma =\sum_{k=0}^3 \psi_k\gamma_k$ and $\sigma^{-1} (1-\sigma)= \delta =\sum_{k=0}^3 \psi_k\delta_k$ where $\gamma_k ,\delta_k\in \mathbb R$ for all $k$. Then   
\begin{align}\label{identity1}
({}^{\psi}_a{\mathcal D}^{\alpha,\sigma}  f )(x,q)  = &e^{-\sum_{k=0}^3 x_k\delta_k} \sigma \ {}^{\psi}\mathcal D[ e^{ \sum_{k=0}^3 x_k\delta_k}  ({}_a{\mathcal I}^{1-\alpha,\sigma} f)(x,q)]   
\end{align} 
and 
\begin{align}\label{identity2}
({}^{\psi}_a{\mathcal D}_r^{\beta,\rho} f)(x,q)  = &e^{-\sum_{k=0}^3 x_k\gamma_k}   \ {}^{\psi}\mathcal D_r[ e^{ \sum_{k=0}^3 x_k\gamma_k}  ({}_a{\mathcal I}^{1-\beta,\rho}  f)(x,q)]\rho,
\end{align} 
for all $f\in AC^1(J_a^b, \mathbb H)$.
\end{prop} 
\begin{proof}
Let us first examine $(\ref{identity1})$. For this purpose, we set
\begin{align*}
 & {}^{\psi}\mathcal D[ e^{ \sum_{k=0}^3 x_k\delta_k}  ({}_a{\mathcal I}^{1-\alpha,\sigma}  f )  (x,q)]  \\
    = &e^{ \sum_{k=0}^3 x_k\delta_k}   \left[  \sigma^{-1} (1-\sigma)   ({}_a{\mathcal I}^{1-\alpha,\sigma}  f)(x,q)
 + {}^{\psi}\mathcal D[ ({}_a{\mathcal I}^{1-\alpha,\sigma}  f )  (x,q) ]   \right] \\
=   & e^{ \sum_{k=0}^3 x_k\delta_k}    \sigma^{-1}  ({}_a{\mathcal D}^{\alpha,\sigma} f)(x,q).
\end{align*} 
The proof of identity $(\ref{identity2})$ runs as before.
\end{proof}

\begin{prop}\label{SBPFP}(Stokes and Borel-Pompeiu formulas induced from ${}^{\psi}_a{\mathcal D}^{\alpha,\sigma}$ and 
${}^{\psi} {\mathcal D}_{r,b}^{\beta,\rho}$). Let $f,g \in AC^1(J_a^b,\mathbb H)$  such that the mappings 
$x\mapsto ({}_a{\mathcal I}^{1-\alpha,\sigma}f)(x,q)$,  $x\mapsto ({\mathcal I}_b^{1-\beta,\rho}  g )(x,q)$
belong to $C^1(J_a^b, \mathbb H)$.

Let $\Omega\subset J_a^b$ be an open bounded domain with boundary $\partial \Omega$ a 3-dimensional smooth surface. Then 
\begin{align*}
 & \int_{\partial \Omega}   ( {\mathcal I}_b^{1-\beta,\rho}  g )  (x,q) \   {}^{\psi}\nu_{x,\gamma,\delta}   \   ({}_a{\mathcal I}^{1-\alpha,\sigma}  f )  (x,q) \\
  =  
   &   \int_{\Omega } \left[   ( {\mathcal I}_b^{1-\beta,\rho}  g )  (x,q)   
  \sigma^{-1}  ({}^{\psi}_a{\mathcal D}^{\alpha,\sigma}  f )(x,q)  \right. \\
& \left.  +  ({}^{\psi} {\mathcal D}_{r,b}^{\beta,\rho}  f )(x,q) \rho^{-1}   ({}_a{\mathcal I}^{1-\alpha,\sigma}  f )  (x,q)
\right] dx_{\gamma,\delta},
\end{align*}
where $ {}^{\psi}\nu_{x,\gamma,\delta} = e^{ \sum_{k=0}^3 x_k(\gamma_k+ \delta_k)} \nu^\psi_x $ and $dx_{\gamma,\delta} =e^{\sum_{k=0}^3 x_k(\gamma_k+\delta_k)}dx.$

About the Borel-Pompieu type formula we have that 
\begin{align*}  
 &    \int_{\partial \Omega} {}_a K^{\alpha,\sigma}_{\psi}(y,x) \ \nu_{y}^{\psi}        \ ({}_a{\mathcal I}^{1-\alpha,\sigma}  f )(y,q)   +
    ( {\mathcal I}_b^{1-\beta, \rho }  g )(y,q) \ \nu_{y}^{\psi} \ K^{\beta,\rho}_{b,\psi}(y,x)         
    \\
  & - 
\sum_{i=0}^3 {}_{a_i}D^{\alpha_i,\sigma_i} \left\{\int_{\Omega} e^{ \sum_{k=0}^3 (y_k-x_k)\delta_k} K_{\psi} (y-x) \sigma^{-1} ({}^{\psi}_a{\mathcal D}^{\alpha,\sigma}f)(y,q) dy  \right\} \nonumber \\
& - 
\sum_{i=0}^3 D_{b_i}^{\beta_i,\rho_i}\left\{ \int_{\Omega}  ({}^{\psi} {\mathcal D}_{r,b}^{\beta,\rho}  f )(y,q)   e^{ \sum_{k=0}^3 (y_k-x_k)\gamma_k}\rho^{-1} K_{\psi} (y-x) dy \right\} \nonumber \\
		=  &  \left\{ 
		\begin{array}{l}
		\displaystyle  \sum_{i=0}^3 (f+g)(q_0,\dots, x_i,\dots, q_3) 
		 +  {}_a N^{\alpha,\sigma}(f,x,q )  +   N_b^{\beta,\rho}(g,x,q ) 
		     , \  x\in \Omega  ,
		 \\ 0 ,   \  x\in \mathbb H\setminus\overline{\Omega},                     
\end{array} \right.  
\end{align*}
where 
$$ {}_aK^{\alpha,\sigma}_{\psi}(y,x) = \sum_{i=0}^3 {}_{a_i}D^{\alpha_i,\sigma_i} [ e^{ \sum_{k=0}^3 (y_k-x_k)\delta_k}K_{\psi}(y-x)],$$
$$K^{\beta,\rho}_{b,\psi}(y,x) = \sum_{i=0}^3  D_{b_i}^{\beta_i,\rho_i}[e^{ \sum_{k=0}^3 (y_k-x_k)\delta_k}K_{\psi}(y-x)],$$
  $${}_a N^{\alpha,\sigma}(f,x,q ) = \sum_{i=0}^3  ({}_{a_i}{ I}^{1-\alpha_i,\sigma_i}  f )(q_0, \dots, x_i , \dots, q_3 ) 		  \left[\sum_{j=0, j\neq i}^3  {}_{a_j}D^{\alpha_j,\sigma_j}(1)  \right]$$
and 
$$ N_b^{\beta,\rho}(g,x,q ) = \sum_{i=0}^3  ({ I}_{b_i}^{1-\beta_i,\rho_i}  g )(q_0, \dots, x_i , \dots, q_3 ) 		  \left[\sum_{j=0, j\neq i}^3  D_{b_j}^{\beta_j,\rho_j}(1) \right].$$
\end{prop} 

\begin{proof}
The quaternionic integral  Stokes' formula for $({\mathcal I}_b^{1-\beta,\rho} g )(x,q) e^{ \sum_{k=0}^3 x_k\gamma_k}$ and $({}_a{\mathcal I}^{1-\alpha,\sigma} f)(x,q)  e^{ \sum_{k=0}^3 x_k\delta_k}$ gives us that 
\begin{align*}
 & \int_{\partial  \Omega}   ( {\mathcal I}_b^{1-\beta,\rho}  g)(x,q) e^{ \sum_{k=0}^3 x_k(\gamma_k+ \delta_k)}  \nu^\psi_x  ({}_a{\mathcal I}^{1-\alpha,\sigma}  f )  (x,q) \\
  =  
   &   \int_{\Omega } \left( e^{ \sum_{k=0}^3 x_k\gamma_k}  ( {\mathcal I}_b^{1-\beta,\rho}  g )  (x,q)   
{}^\psi \mathcal  D[e^{ \sum_{k=0}^3 x_k\delta_k}  ({}_a{\mathcal I}^{1-\alpha,\sigma}  f )  (x,q)
]  \right. \\
& \left.  + {}^{{\psi}}\mathcal  D_r[e^{ \sum_{k=0}^3 x_k\gamma_k}  ( {\mathcal I}_b^{1-\beta,\rho}  g )  (x,q) ] e^{ \sum_{k=0}^3 x_k\delta_k}  ({}_a{\mathcal I}^{1-\alpha,\sigma}  f )  (x,q)
\right)   dx\\
  =  
   &   \int_{\Omega } \left[   ( {\mathcal I}_b^{1-\beta,\rho}  g )  (x,q)   
  \sigma^{-1}  ({}^{\psi}_a{\mathcal D}^{\alpha,\sigma}  f )(x,q)  \right. \\
& \left.  +  ({}^{\psi} {\mathcal D}_{r,b}^{\beta,\rho}  g )(x,q) \rho^{-1}   ({}_a{\mathcal I}^{1-\alpha,\sigma}  f )  (x,q)
\right]   e^{\sum_{k=0}^3 x_k (\gamma_k+\delta_k)} dx,
\end{align*}
where Proposition \ref{DFracPropor} was used.

Hyperholomorphic Borel-Pompieu formula, \eqref{BorelHyp} and Proposition \ref{DFracPropor}  allows to have that 
\begin{align*}  
 &  \int_{\partial \Omega}K_{\psi}(y-x)\nu_{y}^{\psi} e^{ \sum_{k=0}^3 y_k\delta_k}  ({}_a{\mathcal I}^{1-\alpha,\sigma}  f )(y,q)   
  \\
  & - 
\int_{\Omega} K_{\psi} (y-x) e^{\sum_{k=0}^3 y_k\delta_k} \sigma^{-1}    ({}^{\psi}_a{\mathcal D}^{\alpha,\sigma}f)(y,q) dy   \nonumber \\
		=  & 
		 \left\{ 
		 \begin{array}{ll}  e^{ \sum_{k=0}^3 x_k\delta_k}  ({}_a{\mathcal I}^{1-\alpha,\sigma}f)(x,q), &  x\in \Omega,  \\ 0 , &  x\in \mathbb H\setminus\overline{\Omega}                      
\end{array} \right. 
\end{align*} 
Applying on both sides the operator $\sum_{i=0}^3 {}_{a_i}D^{1-\alpha_i,\sigma_i}$, where ${}_{a_i}D^{1-\alpha_i,\sigma_i}$ is the fractional proportional partial derivative in coordinate $x_i$ for $i=0,1,2,3$ yields    
\begin{align*}  
 & \sum_{i=0}^3 {}_{a_i}D^{1-\alpha_i,\sigma_i}  
   \left[
  \int_{\partial \Omega} e^{ \sum_{k=0}^3 (y_k-x_k)\delta_k}K_{\psi}(y-x)\nu_{y}^{\psi}       ({}_a{\mathcal I}^{1-\alpha,\sigma}  f )(y,q)  \right]
  \\
  & - 
\sum_{i=0}^3 {}_{a_i}D^{1-\alpha_i,\sigma_i}  
 \left[ \int_{\Omega} e^{ \sum_{k=0}^3 (y_k-x_k)\delta_k} K_{\psi} (y-x)  \sigma^{-1}    ({}^{\psi}_a{\mathcal D}^{\alpha,\sigma}  f )(y,q) dy  \right] \nonumber \\
		=  &  \left\{
		 \begin{array}{ll} \sum_{i=0}^3 {}_{a_i}D^{1-\alpha_i,\sigma_i}  
		  ({}_a{\mathcal I}^{1-\alpha,\sigma}  f )(x,q)      , &  y\in \Omega,  
		  \\ 0 , &  y\in \mathbb H\setminus\overline{\Omega}.                     
\end{array} \right.      \\		
=  &  \left\{ \begin{array}{l} 
	\displaystyle 	\sum_{i,j=0}^3 {}_{a_i}D^{1-\alpha_i,\sigma_i} \
		  {}_{a_j}I^{1-\alpha_j,\sigma_j} f (q_0,\dots, x_j, \dots q_3)   
		      , \  x\in  \Omega,  \\
		 0 , \  x\in \mathbb H\setminus\overline{\Omega}.                     
\end{array} \right. 
\end{align*}
Using  Leibniz rule and \eqref{FundTheoFPW} then  the previous identity becomes  
\begin{align*}  
 & \int_{\partial \Omega} \sum_{i=0}^3 {}_{a_i}D^{1-\alpha_i,\sigma_i} \left[ e^{ \sum_{k=0}^3 (y_k-x_k)\delta_k}K_{\psi}(y-x)  \right]
  \nu_{y}^{\psi}       ({}_a{\mathcal I}^{1-\alpha,\sigma}  f )(y,q) \\
  & - 
\sum_{i=0}^3 {}_{a_i}D^{1-\alpha_i,\sigma_i}  
 \left[ \int_{\Omega} e^{ \sum_{k=0}^3 (y_k-x_k)\delta_k} K_{\psi} (y-x)  \sigma^{-1}    ({}^{\psi}_a{\mathcal D}^{\alpha,\sigma}  f )(y,q) dy  \right] \nonumber \\
		=  &  \left\{ 
		\begin{array}{l}
\displaystyle 		 \sum_{i=0}^3 f(q_0,\dots, x_i,\dots, q_3) +   
	\\	 \displaystyle    \sum_{i=0}^3  ({}_{a_i}{ I}^{1-\alpha_i,\sigma_i}  f )(q_0, \dots, x_i, \dots, q_3) 
		  \left[ \sum_{j=0, j\neq i}^3  {}_{a_j}D^{1-\alpha_j,\sigma_j}(1)   \right]  
		     ,\   x\in \Omega,  
		 \\ 0 , \   x\in \mathbb H\setminus\overline{\Omega}.                     
\end{array} \right. 
\end{align*}
Note that formula 
\begin{align*}  
 &    \int_{\partial \Omega} 
        ( {\mathcal I}_b^{1-\beta,\rho}  g )(y,q)  \nu_{y}^{\psi} \sum_{i=0}^3 D_{b_i}^{1-\beta_i,\rho_i}\left[ e^{ \sum_{k=0}^3 (y_k-x_k)\lambda_k}K_{\psi}(y-x)  \right]
        \\
  & - 
\sum_{i=0}^3  D_{b_i}^{1-\beta_i,\rho_i}  
 \left[ \int_{\Omega} ({}^{\psi} {\mathcal D}_{r,b}^{\beta,\rho}  g )(y,q)  e^{ \sum_{k=0}^3 (y_k-x_k)\lambda_k}  \rho^{-1}  K_{\psi} (y-x) dy  \right] \nonumber \\
		=  &  \left\{ 
\begin{array}{l}
		\displaystyle  \sum_{i=0}^3 g(q_0,\dots, x_i,\dots, q_3) +
	\\	 \displaystyle    \sum_{i=0}^3  ( { I}_{b_i}^{1-\beta_i,\rho_i}  g )(q_0, \dots, x_i, \dots, q_3) 
		  \left[ \sum_{j=0, j\neq i}^3   D_{b_j}^{1-\beta_j,\rho_j}(1)   \right]  
		     , \  x\in \Omega,  
		 \\ 0 , \  x\in \mathbb H\setminus\overline{\Omega},                     
\end{array} \right. 
\end{align*}
is obtained from similar computation. Combining the previous formulas completes the proof.
\end{proof}
 
We are thus led to the Cauchy type theorem and formula induced from ${}^{\psi}_a{\mathcal D}^{\alpha,\sigma}$ and ${}^{\psi} {\mathcal D}_{r,b}^{\beta,\rho}$.
\begin{cor} \label{CPFP}
Set $f,g \in AC^1(J_a^b,\mathbb H)$ such that the mappings $x\mapsto ({}_a{\mathcal I}^{1-\alpha,\sigma}f)(x,q)$ and  $x\mapsto ( {\mathcal I}_b^{1-\beta,\rho}  g )(x,q)$ belong to $C^1(\overline{J_a^b}, \mathbb H)$. 

Let $\Omega\subset J_a^b$ be an open bounded domain such that  $\partial \Omega$ is  a 3-dimensional smooth surface. If 
$$({}^{\psi}_a{\mathcal D}^{\alpha,\sigma}f)(x,q)= 0 = ({}^{\psi}{\mathcal D}_{r,b}^{\beta,\rho}f)(x,q), \quad \forall x\in J_a^b$$
then 
\begin{align*}
\int_{\partial \Omega}   ( {\mathcal I}_b^{1-\beta,\rho}  g )  (x,q)  \ {}^{\psi}  \nu_{x,\gamma,\delta}  \    ({}_a{\mathcal I}^{1-\alpha,\sigma}f)(x,q) = 0 
\end{align*}
and 
 \begin{align*}  
 &    \int_{\partial \Omega}\left(   K^{\alpha,\sigma}_{\psi}(y,x) \ \nu_{y}^{\psi}     \   ({}_a{\mathcal I}^{1-\alpha,\sigma}  f )(y,q)   +
    ( {\mathcal I}_b^{1-\beta, \rho }  g )(y,q) \ \nu_{y}^{\psi} \  K^{\beta,\rho}_{b,\psi}(y,x)        \right) 
    \\
	=  &  \left\{ 
		\begin{array}{l}
	\displaystyle 	 \sum_{i=0}^3 (f+g)(q_0,\dots, x_i,\dots, q_3) 
		 +   {}_aN^{\alpha,\sigma}(f,x,q )  +   N_b^{\beta,\rho}(g,x,q ) 
		     , \  x\in \Omega,  
		 \\ 0 , \  x\in \mathbb H\setminus\overline{\Omega}.                     
\end{array} \right. 
\end{align*}
\end{cor} 

\begin{rem}\label{remRLC} 
\begin{enumerate}
\item It should be noted that Proposition \ref{SBPFP} and Corollary \ref{CPFP} show a common phenomenon of duality between ${}^{\psi}_a{\mathcal D}^{\alpha,\sigma}$ and ${}^{\psi} {\mathcal D}_{r,b}^{\beta,\rho} $. Analogously, the same phenomenon can be shown in the case of operators: ${}^{\psi}_a{\mathcal D_r}^{\alpha,\sigma}$ and  ${}^{\psi} {\mathcal D}_{b}^{\beta,\rho} $.
\item  Stokes and  Borel-Pompieu formulas induced  ${}^{\psi}_a{\mathcal D}^{\alpha,\sigma}$ and ${}^{\psi} {\mathcal D}_{r,b}^{\beta,\rho} $ and their corollaries and their duality phenomena above presented  can be  also obtained for ${}^{\psi,C}_a{\mathcal D}^{\alpha,\sigma}$ and ${}^{\psi,C} {\mathcal D}_{r,b}^{\beta,\rho}$ using the following relationships between ${}^{\psi}_a{\mathcal D}^{\alpha,\sigma}$ and ${}^{\psi, C}_a{\mathcal D}^{\alpha,\sigma}$ (similar relationships can be obtained for ${}^{\psi} {\mathcal D}_{r,b}^{\beta,\rho} $and ${}^{\psi,C} {\mathcal D}_{r,b}^{\beta,\rho}$):    
Suppose that ${h} \in AC^1(J_a^b,\mathbb H)$ such that $f(x) = {}_{a}\mathcal I^{1-\alpha, \sigma} h(x,q)$ also belongs to $AC^1(J_a^b,\mathbb H)$.
Then  
\begin{align*}
 &   \sum_{i=0}^3{}_{a_i} D^{1-\alpha_i, \sigma_i}[ \ {}_{a}^{\psi, C} \mathcal D^{\alpha, \sigma} f(x,q)] - \sum_{i=0}^3( {}_{a_i}I^{1-\alpha_i, \sigma_i} \ {}^{\psi}
    \mathcal D^{\sigma}f )(q_0, \dots, x_i, \dots, q_3) \lambda_i \\
    =  &\sum_{i=0}^3  {}_a^{\psi} \mathcal D^{\alpha,\sigma} h (q_0, \dots, x_i, \dots, q_3),
\end{align*}
where 
$$\displaystyle \lambda_i =  \sum_{\begin{array}{c}j=0\\ j\neq i \end{array}}{}_{a_j} D^{1-\alpha_j, \sigma_j} (1).$$
If ${h} \in AC^1(J_a^b,\mathbb H)$ such that $f = {}^{\psi}\mathcal D^{\sigma} h $ also belongs to $AC^1(J_a^b,\mathbb H)$ then   
$$\displaystyle {}_a^{\psi}\mathcal D^{\alpha,\sigma} f (x,q) = {}^{\psi} \mathcal D ^{\sigma}[{}_{a}^{\psi, C} \mathcal D^{\alpha, \sigma} h(x,q)],$$  for all $x\in J_a^b$.
\end{enumerate}
\end{rem}

\subsection{Quaternionic proportional fractional Fueter-type operator in Riemann-Louville and Caputo sense with respect to a real valued function}
Let us present the quaternionic version of the fractional proportional derivatives with respect to another function given in \eqref{IntFracPropRespecFuntL}, \eqref{IntFracPropRespecFuntR} and \eqref{DerFracPropRespecFuntL}. 

\begin{defn}
Let $\varphi\in C^1(\overline{J_a^b},\mathbb R)$ such that $(\varphi_i)(x_i):=\varphi(q_0, \dots, x_i, \dots q_3 )$, for all $x_i\in [a_i,b_i]$ and  
$$(\varphi_i)'(x_i):=\displaystyle\frac{\partial }{\partial x_i} \varphi(q_0, \dots, x_i, \dots q_3 ) > 0$$ 
for all $q= \sum_{k=0}^3\psi_k  q_k  \in J_a^b$ and all $x_i \in [a_i,b_i]$ for $i=0,1,2,3$.  The left fractional proportional integral of $ f \in AC^{1}(J_a^b,\mathbb H)$ by coordinates with order $\alpha$,  proportion related to $\sigma$  and  with respect to $\varphi$ is defined by
\begin{align*}
& ({}_a {\mathcal I}^{\alpha,\sigma,\varphi} f) (x,q) \\
= 
 & \sum _{i=0}^3
 \frac{1}{\sigma_i^{\alpha_i} \Gamma(\alpha_i)}
 \int_{a_i}^{x_i}
\dfrac{e^{\frac{\sigma_i-1}{\sigma_i} (\varphi_i(  x_i )-\varphi_i(  \tau_i  )) }    
\ f (q_0, \dots, \tau_i ,\dots, q_3)   \varphi_i'(  \tau_i  )
  }{(\varphi_i( x_i ) -  \varphi_i( \tau_i ))
^{1-{\alpha_i}}} d\tau \nonumber\\
= &\sum _{i=0}^3 ({}_{a_i} I^{\alpha_i,\sigma_i,\varphi_i} f) (q_0, \dots, x_i ,\dots, q_3), 
\end{align*}
Analogously, the right fractional proportional integral of $ f \in AC^{1}(J_a^b,\mathbb H)$ by coordinates with respect to $\varphi$ is     
\begin{align*}
({\mathcal I}_b^{\alpha,\sigma,\varphi} f) (x,q) = \sum _{i=0}^3 (I_{b_i}^{\alpha_i,\sigma_i,\varphi_i} f)(q_0, \dots, x_i ,\dots, q_3). 
\end{align*} 
The $\psi-$quaternionic  left  and   right  fractional proportional  Fueter operator with respect to $\varphi$ in the Riemann-Liouville sense  are given by  
\begin{align}\label{HDerFracPropRespecFuntL}
 {}^{\psi}_a\mathcal D^{\alpha, \sigma, \varphi }[f]  (x,q)  = &  
   {}^{\psi}\mathcal D^{ \sigma, \varphi} [({}_a{\mathcal I}^{1-\alpha,  \sigma, \varphi}  f )  (x,q)  ] 
\end{align} 
 and 
\begin{align}\label{HDerFracPropRespecFuntR}
 {}^{\psi}\mathcal D_b^{\alpha, \sigma, \varphi }[   f]  (x,q)  =  &
 {}^{\psi}\mathcal D^{\sigma, \varphi } [  ( {\mathcal I}_b^{1-\alpha,  \sigma, \varphi}  f )  (x,q) ],   
\end{align} 
respectively, where
\begin{align*}
   {}^{\psi}\mathcal D^{\sigma, \varphi}[h](x):=     (1-\sigma)  h(x) +
 \sigma   \left[ \sum_{k=0}^3(\varphi_k)'(x_k) \right]^{-1}
       ({}^{\psi}\mathcal D  h)(x), 
\end{align*}
for all $h\in C^1(J_a^b,\mathbb H)$.
The right versions of the previous operators are the following: 
\begin{align}\label{HDerFracPropRespecFuntLL}
  {}^{\psi}_a\mathcal D_r^{\alpha, \sigma, \varphi }[   f]  (x,q)  =  & 
     {}^{\psi} \mathcal D_r^{\sigma, \varphi } [   ({}_a{\mathcal I}^{1-\alpha,  \sigma, \varphi}  f )  (x,q) ] , \nonumber \\
   {}^{\psi}\mathcal D_{r,b}^{\alpha, \sigma, \varphi }[   f]  (x,q)  =  &
     {}^{\psi}\mathcal D_r^{\sigma,\varphi} [  ({\mathcal I}_b^{1-\alpha,  \sigma, \varphi}  f )  (x,q)],   
\end{align} 
respectively, where
\begin{align*}
{}^{\psi}\mathcal D_r^{\sigma, \varphi}[h](x):= h(x) (1-\sigma) + ({}^{\psi}\mathcal D_r  h)(x) \left[ \sum_{k=0}^3(\varphi_k)'(x_k) \right]^{-1} \sigma. 
\end{align*}
\end{defn}
\begin{defn}\label{FracPropRespecFuntLCaputo} 
The  $\psi-$quaternionic left and right fractional proportional Fueter operator with respect to $\varphi$ in the Caputo  sense  are given by  
\begin{align*}
 {}^{\psi,C}_a\mathcal D^{\alpha, \sigma, \varphi }[f]  (x,q)  = &    
    ({}_a{\mathcal I}^{1-\alpha,  \sigma, \varphi} \ {}^{\psi}\mathcal D^{ \sigma, \varphi } [ f] ) (x,q) 
 \end{align*} 
 and 
\begin{align*}
 {}^{\psi,C}\mathcal D_b^{\alpha, \sigma, \varphi }[   f]  (x,q)  =  &
   ( {\mathcal I}_b^{1-\alpha,  \sigma, \varphi} \  {}^{\psi}\mathcal D^{\sigma, \varphi } [f] )  (x,q) ,   
\end{align*} 
respectively.   
The right versions of the previous operators are the following: 
\begin{align*}
  {}^{\psi,C}_a\mathcal D_r^{\alpha, \sigma, \varphi }[   f]  (x,q)  =  & 
  ({}_a{\mathcal I}^{1-\alpha,  \sigma, \varphi} \ {}^{\psi}\mathcal D_r^{ \sigma, \varphi } [ f] )  (x,q) , \nonumber \\
   {}^{\psi,C}\mathcal D_{r,b}^{\alpha, \sigma, \varphi }[   f]  (x,q)  =   & 
   ( {\mathcal I}_b^{1-\alpha,  \sigma, \varphi} \ {}^{\psi}\mathcal D_r^{\sigma, \varphi } [f] )  (x,q) .   
\end{align*} 
\end{defn}

\begin{rem}  
We will define a quaternionic operator analogous to that given in \eqref{AxuliarR} and we will extend  the  relationship \eqref{AxuliarComputationsR} to achieve our goal.

\begin{align*}   &  {}^{\psi}_a\mathcal E^{\alpha, \sigma, \varphi }[   f]  (x,q)  \\
= & 
    \sum_{k=0}^3(\varphi_k)'(x_k)
 \sigma^{-1} (1-\sigma)  ({}_a{\mathcal I}^{1-\alpha,  \sigma, \varphi}  f )  (x,q)  + {}^{\psi}\mathcal D ({}_a{\mathcal I}^{1-\alpha,  \sigma, \varphi}  f )  (x,q) , \\   
 &  {}^{\psi}\mathcal E_b^{\alpha, \sigma, \varphi }[   f]  (x,q) \\
 =  & 
    \sum_{k=0}^3(\varphi_k)'(x_k)
 \sigma^{-1} (1-\sigma)  ( {\mathcal I}_b^{1-\alpha,  \sigma, \varphi}  f )  (x,q)  + {}^{\psi}\mathcal D ({\mathcal I}_b^{1-\alpha,  \sigma, \varphi}  f )  (x,q) ,\\
&   {}^{\psi}_a\mathcal E_r^{\alpha, \sigma, \varphi }[   f]  (x,q) \\
 =  & 
   ({}_a{\mathcal I}^{1-\alpha,  \sigma, \varphi}  f )  (x,q)   (1-\sigma)  \sigma^{-1}   \sum_{k=0}^3(\varphi_k)'(x_k)
     + {}^{\psi}\mathcal D_r ({}_a{\mathcal I}^{1-\alpha,  \sigma, \varphi}  f )  (x,q),\\
&  {}^{\psi}\mathcal E_{r,b}^{\alpha, \sigma, \varphi }[   f]  (x,q)  \\
= &  
    ( {\mathcal I}_b^{1-\alpha,  \sigma, \varphi}f)(x,q) (1-\sigma) \sigma^{-1} \sum_{k=0}^3(\varphi_k)'(x_k) + {}^{\psi}\mathcal D_r ({\mathcal I}_b^{1-\alpha,  \sigma, \varphi} f)(x,q) .   
\end{align*} 
We will suppose that there exist $\lambda_k \in  C^1(  [a_k, b_k],  \mathbb R)$ for $k=0,1,2,3$ such that  
$$\sum_{k=0}^3 \psi_k \frac{\partial {\lambda_k}}{\partial x_k} (x_k) = \sum_{k=0}^3(\varphi_k)'(x_k)\sigma^{-1} (1-\sigma).$$ 
Then 
\begin{align}\label{AuxHDerFracPropRespecFuntLYD}
&  e^{- \sum_{k=0}^3 \lambda_k(x_k)} {}^{\psi}\mathcal D[ e^{ \sum_{k=0}^3\lambda_k(x_k)}  ({}_a{\mathcal I}^{1-\alpha,\sigma, \varphi}  f )  (x,q)] \nonumber  \\
=  & \left\{
  \sum_{k=0}^3 \psi_k  \frac{\partial  \lambda_k }{\partial x_k}   ({}_a{\mathcal I}^{1-\alpha,\varphi}  f )  (x,q)  + {}^{\psi}\mathcal D ({}_a{\mathcal I}^{1-\alpha,\sigma,\varphi}  f )  (x,q)     \right\} \nonumber  \\
=  & \left\{
    \sum_{k=0}^3(\varphi_k)'(x_k)
 \sigma^{-1} (1-\sigma) ({}_a{\mathcal I}^{1-\alpha,\sigma, \varphi}  f )  (x,q)  + {}^{\psi}\mathcal D ({}_a{\mathcal I}^{1-\alpha,\sigma,\varphi}f)(x,q) \right\}  \nonumber  \\
 =  &   
  {}^{\psi}_a\mathcal E ^{\alpha,\sigma,\varphi}[f](x,q)  . 
  \end{align}
Analogously, if $\mu_k  \in C^1([a_k,b_k] ,  \mathbb R)$  for $k=0,1,2,3$ such that  
$$\sum_{k=0}^3 \psi_k \frac{\partial {\mu_k}}{\partial x_k} (x_k) =  
 \sum_{k=0}^3(\varphi_k)'(x_k)
  (1-\sigma) \sigma^{-1}.$$ 
Then 
\begin{align}\label{AuxHDerFracPropRespecFuntRYD}
{}^{\psi}_a\mathcal E_r^{\alpha,\sigma,\varphi}[f](x,q)   
= e^{ - \sum_{k=0}^3  \mu_k(x_k)} {}^{\psi}\mathcal D_r [ e^{ \sum_{k=0}^3\mu_k(x_k)} ({}_a{\mathcal I}^{1-\alpha,\sigma, \varphi}f)(x,q)].
\end{align}
\end{rem}
Stokes and Borel-Pompeiu formulas induced from ${}^{\psi}_a{\mathcal D}^{\alpha,\sigma, \varphi}$ and ${}^{\psi}{\mathcal D}_{r,b}^{\beta,\rho,\vartheta}$ are given as follows. 
\begin{prop}\label{SBPFPCF}
Let $\varphi, \vartheta \in C^1(J_a^b  ,\mathbb R)$ such that  
$$(\varphi_i)'(x_i):= \displaystyle\frac{\partial }{\partial x_i} \varphi(q_0, \dots, x_i, \dots q_3 )> 0$$ and $(\vartheta_i)'(x_i):=\displaystyle\frac{\partial }{\partial x_i} \vartheta (q_0, \dots, x_i, \dots q_3 ) > 0$  for all $q= \sum_{k=0}^3\psi_k  q_k  \in J_a^b$ and all $x_i \in [a_i,b_i]$ for  $i=0,1,2,3$. 

Suppose $\lambda_k, \mu_k  \in C^1( [a_k,b_k] , \mathbb R)$ for $k=0,1,2,3$ such that  
$$\sum_{k=0}^3 \psi_k \frac{\partial {\lambda_k}}{\partial x_k} (x_k) = \sum_{k=0}^3(\varphi_k)'(x_k)\sigma^{-1} (1-\sigma)$$ 
and 
$$\sum_{k=0}^3 \psi_k \frac{\partial {\mu_k}}{\partial x_k} (x_k) = \sum_{k=0}^3(\vartheta_k)'(x_k) (1-\rho)\rho^{-1}.$$ 
Let $\Omega\subset J_a^b$ be an open bounded domain with boundary $\partial \Omega$ a 3-dimensional smooth surface. 

Let $f,g \in AC^1(J_a^b,\mathbb H)$ such that the mappings $$x\mapsto ({}_a{\mathcal I}^{1-\alpha,\sigma, \varphi}f)(x,q), \quad x\mapsto ({\mathcal I}_b^{1-\beta,\rho, \vartheta}g)(x,q)$$ belong to $C^1( J_a^b, \mathbb H).$ Then
\begin{align*}
 & \int_{\partial \Omega} ({\mathcal I}_b^{1-\beta,\rho,\vartheta}g)(x,q) \  {}^\psi\nu_x^{\mu,\lambda}   \  ({}_a{\mathcal I}^{1-\alpha,\sigma, \varphi}f)(x,q) \\
  =  
   &   \int_{\Omega } \left\{   ( {\mathcal I}_b^{1-\beta,\rho,\vartheta}  g )  (x,q)   
    \left[ \sum_{k=0}^3(\varphi_k)'(x_k)
 \right]  \sigma ^{-1}  ({}^{\psi}_a{\mathcal D}^{\alpha,\sigma,\varphi}  f )(x,q)   \right. \\
& \left.  +  ({}^{\psi} {\mathcal D}_{r,b}^{\beta,\rho,\vartheta}  g )(x,q)  \rho^{-1} 
 \left[ \sum_{k=0}^3(\vartheta_k)'(x_k) \right] ({}_a{\mathcal I}^{1-\alpha,\sigma,\varphi}  f )  (x,q)
\right\}    dx^{\mu,\lambda},
\end{align*}  
where  ${}^\psi\nu_x^{\mu, \lambda}:= e^{\sum_{k=0}^3 \mu_k(x_k) + \lambda_k(x_k)}\nu^\psi_x $ and $dx^{\mu,\lambda} := e^{\sum_{k=0}^3 \mu_k(x_k)+\lambda_k(x_k)}dx$. 

The Borel Pompeiu type formula is given by
\begin{align*}  
 &  \int_{\partial \Omega} \left[ {}_a K_{\psi}^{\alpha, \sigma, \varphi}(y,x)
 \nu_{y}^{\psi} 
  ({}_a{\mathcal I}^{1-\alpha,\sigma, \varphi}  f )  (y,q)   
  +
  ( {\mathcal I}_b^{1-\beta,\rho, \vartheta}  g )  (y,q)    \nu_{y}^{\psi}   K_{b,\psi}^{\beta, \rho, \vartheta}(y,x) \right]
  \\
  & - \sum_{i=0}^3 {}_{a_i}D^{1-\alpha_i,\sigma_i, \varphi_i} \left\{
\int_{\Omega}  \mathcal H_\psi^{\lambda, \varphi} (y,x)   \sigma ^{-1}({}^{\psi}_a{\mathcal D}^{\alpha,\sigma,\varphi}  f )(y,q)    
   dy  \right\}  \\
   & - \sum_{i=0}^3 D_{b_i}^{1-\beta_i,\rho_i, \vartheta_i} \left\{
\int_{\Omega}       
  ({}^{\psi} {\mathcal D}_{r,b}^{\beta,\rho,\vartheta}  g )(y,q)  \rho^{-1} 
 \mathcal H_\psi^{\rho, \vartheta} (y,x)   dy  \right\}  \\
		=  &  \left\{ 
		\begin{array}{l}
		\displaystyle  \sum_{i=0}^3 (f+g)(q_0,\dots, x_i,\dots, q_3) + {}_aN^{\alpha, \sigma, \varphi}(f,x,q)   
		  + N_{b}^{\beta, \rho, \vartheta}(f,x,q)   ,  \   x\in \Omega, 
		 \\ 0 ,  \   x\in \mathbb H\setminus\overline{\Omega},                     
\end{array} \right. 
\end{align*}
where
\begin{align*}
 {}_aK_{\psi}^{\alpha, \sigma, \varphi}(y,x) : = &  \sum_{i=0}^3 {}_{a_i}D^{1-\alpha_i,\sigma_i, \varphi_i} 
\left[ K_{\psi}(y-x)  e^{ \sum_{k=0}^3(\lambda_k(y_k)-\lambda_k(x_k) )}\right], \\ 
  K_{b,\psi}^{\beta, \rho, \vartheta}(y,x) := &  \sum_{i=0}^3  D_{b_i}^{1-\beta_i,\rho_i, \vartheta_i} 
 \left[ K_{\psi}(y-x)  e^{ \sum_{k=0}^3(\lambda_k(y_k)-\lambda_k(x_k) )}\right], \\ 
 {}_aN^{\alpha, \sigma, \varphi}(f,x,q) = &  \sum_{i=0}^3  ({}_{a_i}{ I}^{1-\alpha_i,\sigma_i, \varphi_i}f)(q_0, \dots, x_i, \dots, q_3) \sum_{j=0, j\neq i}^3  {}_{a_j}D^{1-\alpha_j,\sigma_j, \varphi_j}(1)   , \\
 N_{b}^{\beta, \rho, \vartheta}(g,x,q) = & \sum_{i=0}^3  ( { I}_{ b_i}^{1-\beta_i,\rho_i, \vartheta_i}  g )(q_0, \dots, x_i, \dots, q_3) 
		   \sum_{j=0, j\neq i}^3   D_{b_j}^{1-\beta_j,\rho_j, \vartheta_j}(1)  , \\
 \mathcal H_\psi^{\lambda, \varphi} (y,x) = & K_{\psi} (y-x)     
  e^{   \sum_{k=0}^3 (\lambda_k(y_k) - \lambda_k(x_k) ) }  \left[ \sum_{k=0}^3(\varphi_k)'(y_k)\right], \\
  \mathcal H_\psi^{\mu, \vartheta} (y,x)= &  K_{\psi} (y-x)    e^{   \sum_{k=0}^3 (\mu_k(y_k) - \mu_k(x_k) ) }  \left[ \sum_{k=0}^3(\vartheta_k)'(y_k)
 \right]         ,
\end{align*}
     for all $x \in J_a^b$.  Here ${}_{a_i}D^{1-\alpha_i,\sigma_i,\varphi_i}$ and $D_{b_i}^{1-\beta_i,\rho_i,\vartheta_i}$ stand for the fractional proportional partial derivative in coordinate $x_i$  for $i=0,1,2,3$.
\end{prop} 

\begin{proof}
Combining hyperholomorphic Stokes'formula with \eqref{AuxHDerFracPropRespecFuntLYD} and \eqref{AuxHDerFracPropRespecFuntRYD} we obtain 
\begin{align*}
 & \int_{\partial J_a^b}   ({}_a{\mathcal I}^{1-\beta,\rho,\vartheta}  g )  (x,q)     e^{ \sum_{k=0}^3 \mu_k(x_k)+ \lambda_k(x_k)}  \nu^\psi_x     ({}_a{\mathcal I}^{1-\alpha,\sigma, \varphi}  f )  (x,q) \\
  =  
   &   \int_{\Omega } \left(   ({}_a{\mathcal I}^{1-\beta,\rho,\vartheta}  g )  (x,q)   
    ({}^{\psi}_a{\mathcal E}^{\alpha,\sigma,\varphi}  f )(x,q)  \right. \\
& \left.  +  ({}^{\psi}_a{\mathcal E}_r^{\beta,\rho,\vartheta}  g )(x,q)   ({}_a{\mathcal I}^{1-\alpha,\sigma,\varphi}  f )  (x,q)
\right)   e^{\sum_{k=0}^3   \mu_k(x_k)+\lambda_k(x_k) } dx.
\end{align*}
Note that  
\begin{align}\label{EtoD} 
({}^{\psi}_a{\mathcal E}^{\alpha,\sigma,\varphi}  f )(x,q) = &   \left[ \sum_{k=0}^3(\varphi_k)'(x_k)
 \right]  \sigma ^{-1}({}^{\psi}_a{\mathcal D}^{\alpha,\sigma,\varphi}  f )(x,q)    ,  \nonumber    \\
({}^{\psi}_a{\mathcal E}_r^{\beta,\rho,\vartheta}  g )(x,q) = &  ({}^{\psi}_a{\mathcal D}_r^{\beta,\rho,\vartheta}  g )(x,q)  \rho^{-1} 
 \left[ \sum_{k=0}^3(\vartheta_k)'(x_k)
 \right]       
\end{align}
to see that 
\begin{align*}
 & \int_{\partial J_a^b}   ({}_a{\mathcal I}^{1-\beta,\rho,\vartheta}  g )  (x,q)     e^{ \sum_{k=0}^3 \mu_k(x_k)+ \lambda_k(x_k)}  \nu^\psi_x     ({}_a{\mathcal I}^{1-\alpha,\sigma, \varphi}  f )  (x,q) \\
  =  
   &   \int_{\Omega } \left\{   ({}_a{\mathcal I}^{1-\beta,\rho,\vartheta}  g )  (x,q)   
    \left[ \sum_{k=0}^3(\varphi_k)'(x_k)
 \right]  \sigma ^{-1}  ({}^{\psi}_a{\mathcal D}^{\alpha,\sigma,\varphi}  f )(x,q) +  \right. \\
& \left.  \quad  ({}^{\psi}_a{\mathcal D}_r^{\beta,\rho,\vartheta}  g )(x,q)  \rho^{-1} 
 \left[ \sum_{k=0}^3(\vartheta_k)'(x_k) \right] ({}_a{\mathcal I}^{1-\alpha,\sigma,\varphi}  f )  (x,q)  
\right\} \times \\
&   \quad  e^{\sum_{k=0}^3   \mu_k(x_k)+\lambda_k(x_k) } dx.
\end{align*}
Applying quaternionic Borel-Pompieu formula and \eqref{AuxHDerFracPropRespecFuntLYD} yields 
\begin{align*}  
 &  \int_{\partial \Omega}K_{\psi}(y-x)\nu_{y}^{\psi}   e^{ \sum_{k=0}^3(\lambda_k(y_k)-\lambda_k(x_k) )}  ({}_a{\mathcal I}^{1-\alpha,\sigma, \varphi}  f )  (y,q)   
  \\
  & - 
\int_{\Omega}  K_{\psi} (y-x)     
  e^{   \sum_{k=0}^3 (\lambda_k(y_k) - \lambda_k(x_k) ) } {}^{\psi}_a\mathcal E ^{\alpha,\sigma,\varphi} [ f ](y,q)   
   dy   \nonumber \\
		=  & 
		 \left\{ 
		 \begin{array}{ll}  
		    ({}_a{\mathcal I}^{1-\alpha,\sigma, \varphi}  f )  (x,q)     , &  x\in J_a^b,  \\ 
		  0 , &  x\in \mathbb H\setminus\overline{J_a^b}.                     
\end{array} \right. 
\end{align*} 
Using \eqref{EtoD} and applying $\sum_{i=0}^3 {}_{a_i}D^{1-\alpha_i,\sigma_i, \varphi_i}$ on both sides sides of the previous formula we obtain 
\begin{align*}  
 &  \int_{\partial \Omega} \sum_{i=0}^3 {}_{a_i}D^{1-\alpha_i,\sigma_i, \varphi_i} 
 \left[ K_{\psi}(y-x)  e^{ \sum_{k=0}^3(\lambda_k(y_k)-\lambda_k(x_k) )}\right] 
 \nu_{y}^{\psi} 
  ({}_a{\mathcal I}^{1-\alpha,\sigma, \varphi}  f )  (y,q)   
  \\
  & - \sum_{i=0}^3 {}_{a_i}D^{1-\alpha_i,\sigma_i, \varphi_i} \{
\int_{\Omega}  K_{\psi} (y-x)     
  e^{\sum_{k=0}^3 (\lambda_k(y_k) - \lambda_k(x_k))}     \left[ \sum_{k=0}^3(\varphi_k)'(y_k)
 \right]  \times     \\
 &   \  \  \  \ 
    \sigma ^{-1}({}^{\psi}_a{\mathcal D}^{\alpha,\sigma,\varphi}  f )(y,q)     
   dy  \}  \nonumber \\
		=  &  \left\{ 
		\begin{array}{l}
		 \sum_{i=0}^3 f(q_0,\dots, x_i,\dots, q_3) + \\
	    \sum_{i=0}^3  ({}_{a_i}{ I}^{1-\alpha_i,\sigma_i, \varphi_i}  f )(q_0, \dots, x_i, \dots, q_3) 
		  (\sum_{j=0, j\neq i}^3  {}_{a_j}D^{1-\alpha_j,\sigma_j, \varphi_j}[1]   )  
		     , \  x\in \Omega  
		 \\ 0 , \  x\in \mathbb H\setminus\overline{\Omega},                     
\end{array} \right. 
\end{align*} 
where Leibniz rule and \eqref{FundTheoFPW} are used. Therefore,
\begin{align*}  
 &  \int_{\partial \Omega}{}_a K_{\psi}^{\alpha, \sigma, \varphi}(y,x)
  \  \nu_{y}^{\psi}  \
  ({}_a{\mathcal I}^{1-\alpha,\sigma, \varphi}  f )  (y,q)   
  \\
  & - \sum_{i=0}^3 {}_{a_i}D^{1-\alpha_i,\sigma_i, \varphi_i}  \{
\int_{\Omega}  K_{\psi} (y-x)     
  e^{   \sum_{k=0}^3 (\lambda_k(y_k) - \lambda_k(x_k) ) }  \left[ \sum_{k=0}^3(\varphi_k)'(y_k)
 \right] \times  \\
&  \ \ \ \   \sigma ^{-1}({}^{\psi}_a{\mathcal D}^{\alpha,\sigma,\varphi}  f )(y,q)    
   dy  \}  \\
		=  &  \left\{ 
		\begin{array}{l}
		 \sum_{i=0}^3 f(q_0,\dots, x_i,\dots, q_3) + {}_aN^{\alpha, \sigma, \varphi}(f,x,q)   
		     ,   x\in \Omega, 
		 \\ 0 ,   x\in \mathbb H\setminus\overline{\Omega}.                     
\end{array} \right. 
\end{align*} 
Formula 
  \begin{align*}  
 &  \int_{\partial \Omega}  
  ( {\mathcal I}_b^{1-\beta,\rho, \vartheta}  g )  (y,q)   \   \nu_{y}^{\psi} \   K_{b,\psi}^{\beta, \rho, \vartheta}(y,x)
  \\
   & - \sum_{i=0}^3  D_{b_i}^{1-\beta_i,\rho_i, \vartheta_i}   \{
\int_{\Omega}       
  ({}^{\psi} {\mathcal D}_{r,b}^{\beta,\rho,\vartheta}  g )(y,q)  \rho^{-1} 
 \left[ \sum_{k=0}^3(\vartheta_k)'(y_k)
 \right]    \times     
      \\
&  \ \ \ \  e^{   \sum_{k=0}^3 (\mu_k(y_k) - \mu_k(x_k) ) } K_{\psi} (y-x)    dy   \}  \\
		=  &  \left\{ 
		\begin{array}{l}
		 \sum_{i=0}^3 g(q_0,\dots, x_i,\dots, q_3)  
		  + N_b^{\beta, \rho, \vartheta}(f,x,q)   , \  x\in \Omega, 
		 \\ 0 , \  x\in \mathbb H\setminus\overline{\Omega}.                     
\end{array} \right. 
\end{align*} 
is obtained from similar reasoning and using right version of the operators used for $f$. The desired conclusion follows summing the two previous identities. 
\end{proof}
Finally, we present the Cauchy type theorem and formula induced from ${}^{\psi}_a{\mathcal D}^{\alpha,\sigma, \varphi}$ and 
${}^{\psi}{\mathcal D}_{r,b}^{\beta,\rho, \vartheta}$.
\begin{cor} \label{CORSBPFPCF} 
Set $f,g \in AC^1(J_a^b,\mathbb H)$  such that the mappings $$x\mapsto ({}_a{\mathcal I}^{1-\alpha,\sigma, \varphi}f)(x,q), \quad  x\mapsto ({\mathcal I}_b^{1-\beta,\rho, \vartheta}g)(x,q)$$ belong to $C^1(J_a^b, \mathbb H)$. If 
$$({}^{\psi}_a{\mathcal D}^{\alpha,\sigma, \varphi}f)(x,q)= 0 =({}^{\psi}{\mathcal D}_{r,b}^{\beta,\rho, \vartheta}f)(x,q), \quad \forall x\in J_a^b$$
then 
\begin{align*}
 & \int_{\partial \Omega}({\mathcal I}_b^{1-\beta,\rho,\vartheta}  g )  (x,q)     \nu^\psi_x (\mu,\lambda)    ({}_a{\mathcal I}^{1-\alpha,\sigma, \varphi}f)(x,q)=0 
\end{align*} 
and 
\begin{align*}  
 &  \int_{\partial \Omega} \left( {}_a K_{\psi}^{\alpha, \sigma, \varphi}(y,x)
 \nu_{y}^{\psi} 
  ({}_a{\mathcal I}^{1-\alpha,\sigma, \varphi}  f )  (y,q)   
  +
  ( {\mathcal I}_b^{1-\beta,\rho, \vartheta}  g )  (y,q)    \nu_{y}^{\psi}   K_{b,\psi}^{\beta, \rho, \vartheta}(y,x) \right)
  \\
		=  &  \left\{ 
		\begin{array}{l}
		 \sum_{i=0}^3 (f+g)(q_0,\dots, x_i,\dots, q_3) + {}_aN^{\alpha, \sigma, \varphi}(f,x,q)   
		  + N_{b}^{\beta, \rho, \vartheta}(f,x,q)   , \  x\in \Omega, 
		 \\ 0 , \  x\in \mathbb H\setminus\overline{\Omega},                     
\end{array} \right. 
\end{align*}
for all $x\in J_a^b$.   
\end{cor} 
\begin{rem} 
In Proposition \ref{SBPFPCF} and Corollary \ref{CORSBPFPCF} the duality phenomenon between operators ${}^{\psi}_a{\mathcal D}^{\alpha,\sigma, \varphi}$ and ${}^{\psi}{\mathcal D}_{r,b}^{\beta,\rho, \vartheta}$ is presented and from similar computations we can obtain the duality in these statements for ${}^{\psi}_a{\mathcal D}_r^{\alpha,\sigma, \varphi} $ and ${}^{\psi}{\mathcal D}_{b}^{\beta,\rho, \vartheta} $.
The following facts quite clearly indicate that the fractional differential operators were established preserving the structure of the real case and a natural order among them.     
\begin{enumerate}
\item If $\varphi (x)= x_0+x_1+x_2+x_3$ then $4({}^{\psi}_a{\mathcal D}^{\alpha,\sigma, \varphi}f)(x,q) = ({}^{\psi}_a{\mathcal D}^{\alpha,\sigma}f)(x,q) + 3 (1-\sigma)   ({}_a{\mathcal I}^{1-\alpha,\sigma}f)(x,q)$.
\item If $\varphi (x)= x_0+x_1+x_2+x_3$ and $\sigma_0=\sigma_1= \sigma_2 =\sigma_3 =1$ then $({}^{\psi}_a{\mathcal D}^{\alpha,\sigma, \varphi}f)(x,q)= (1- \psi_0 -\psi_1-\psi_2-\psi_3)({}_a \mathcal I^{\alpha} f)(x,q) + \frac{1}{4}(\psi_0+ \psi_1+\psi_2 + \psi_3) ({}^{\psi}_a{\mathcal D}^{\alpha}f)(x,q)$, where $({}^{\psi}_a{\mathcal D}^{\alpha}f)(x,q)$  is the quaternionic fractional $\psi-$Fueter type operator studied in \cite{GB} and extended with respect to another vector valued function in \cite{BG2}.
\item If $\varphi(x) =\frac{ x_0^{\mu}}{\mu} + \frac{ x_1^{\mu}}{\mu}+\frac{ x_2^{\mu}}{\mu}+\frac{x_3^{\mu}}{\mu}$ and $\sigma_0=\sigma_1= \sigma_2 =\sigma_3 =1$ then the real components of our quaternionic fractional operators are related with the fractional partial derivatives in the Katugampola setting.
\item If $\varphi(x) =\ln(x_0) + \ln(x_1) + \ln(x_2)+ \ln(x_3)$ and $\sigma_0=\sigma_1= \sigma_2 =\sigma_3 =1$ then the real components of our quaternionic fractional operators are related with the Hadamard fractional partial derivatives.
\end{enumerate}
\end{rem}

\begin{rem} 
The  quaternionic fractional operators ${}^{\psi}_a\mathcal D^{\alpha, \sigma, \varphi}$ and  ${}^{\psi,C}_a\mathcal D^{\alpha, \sigma, \varphi}$  satisfy  similar relationships  to  presented in the  second fact of Remark \ref{remRLC} between    
${}^{\psi}_a\mathcal D^{\alpha, \sigma } $ and  ${}^{\psi,C}_a\mathcal D^{\alpha, \sigma}$, also the same behaviour is presented for   ${}^{\psi} \mathcal D_b^{\alpha, \sigma ,\varphi}$ and  ${}^{\psi,C} \mathcal D_b^{\alpha, \sigma ,\varphi}$.
So using these identities we can obtain the version of all results given in  this subsection for  the quaternionic fractional operators in Caputo sense given in Definition \ref{FracPropRespecFuntLCaputo}.
\end{rem}

\section*{Acknowledgments}
The authors wish to thank the Instituto Polit\'ecnico Nacional (grant numbers SIP20232103, SIP20230312) for partial support.

\end{document}